\documentclass{amsart}

\usepackage{amsfonts,amssymb,amsmath,amsthm,latexsym,booktabs,lscape,color,enumerate, amscd}

\theoremstyle{definition}
\newtheorem{definition}{Definition}[section]
\newtheorem{remark}[definition]{Remark}
\newtheorem{example}[definition]{Example}

\theoremstyle{plain}

\newtheorem{proposition}[definition]{Proposition}
\newtheorem{theorem}[definition]{Theorem}
\newtheorem{corollary}[definition]{Corollary}

\newtheorem{notation}[definition]{Notation}
\newtheorem{defis}[definition]{Definitions}

\newcommand{\Z}{\mathrm{Z}}
\newcommand{\R}{\mathcal{R}}
\newcommand{\M}{\mathrm{M}}
\newcommand{\T}{\mathcal{T}}
\newcommand{\Hc}{\mathcal{H}}
\newcommand{\N}{\mathcal{N}}
\newcommand{\Tf}{\mathfrak{T}}
\newcommand{\A}{\mathcal{A}}

\newcommand{\B}{\mathrm{B}}
\newcommand{\NN}{\mathbb{N}}

\newcommand{\Tri}{\mathrm{Trian \,}}
\newcommand{\Id}{\mathrm{Id}}

\usepackage{color}



\begin{document}

\title{$\sigma$-mappings of triangular algebras}

\author[S\'anchez-Ortega]{Juana S\'anchez-Ortega}

\address{Departamento de \'Algebra, Geometr\'ia y Topolog\'ia, Universidad de M\'alaga, M\'alaga, Spain}

\email{jsanchezo@uma.es}

\maketitle


\begin{abstract}
Let $\A$ be an algebra and $\sigma$ an automorphism of $\A$. A linear map $d$ of $\A$ is called a $\sigma$-derivation of $\A$ if $d(xy) = d(x)y + \sigma(x)d(y)$, for all $x, y \in \A$. A linear map $D$ is said to be a generalized $\sigma$-derivation of $\A$ if there exists a $\sigma$-derivation $d$ of $\A$ such that $D(xy) = D(x)y + \sigma(x)d(y)$, for all $x, y \in \A$. An additive map $\Theta$ of $\A$ is $\sigma$-centralizing if $\Theta(x)x - \sigma(x)\Theta(x) \in \Z(\A)$, for all $x \in \A$. In this paper, precise descriptions of generalized $\sigma$-derivations and $\sigma$-centralizing maps of triangular algebras are given. Analogues of the so-called commutative theorems, due to Posner and Mayne, are also proved for the triangular algebra setting. 
\end{abstract}


\section{Introduction}

Triangular algebras were introduced by Chase \cite{Cha} in the early 1960s. The study of derivations and related maps of triangular algebras have attracted the attention of a number of authors in the past few years. The origin of this theory can be placed in the middle/late 1990s, when several authors undertook the study of derivations and related maps over some particular families of triangular algebras (see, for example, \cite{Chr, CM, FM, Jn, Zh} and references therein). Motivated by those works Cheung \cite{Ch1} initiated, in his thesis, the study of linear maps of (abstract) triangular algebras. Cheung's research has inspired several authors to investigate many distinct maps of triangular algebras.

A map $\Theta$ of an algebra $\A$ into itself is said to be {\bf commuting} if $\Theta(x)$ commutes with $x$, for every $x \in \A$. The first important result on commuting maps dates back to 1957; it is due to Posner \cite{Po}. The classical result of Posner states that zero is the only commuting derivation of a noncommutative prime ring. Very recently,
Repka and the author \cite{RSO} have proved that Posner's theorem remains true in the triangular algebra setting. 
At this point, it should be pointed out that Posner proved a more general result. More precisely, he proved that if $d$ is a derivation of a noncommutative ring $R$ such that $[d(x), x] \in \Z(R)$, for all $x \in R$, then $d = 0$; maps satisfying the condition above are called {\bf centralizing maps}. Therefore, it is natural to consider the question of whether every centralizing derivation of a triangular algebra is the zero map. A more general question will be investigated in Section \ref{thms}.

In the late 1970s, Mayne \cite{M} proved an analogue of Posner's theorem for centralizing automorphisms of prime rings. Several authors have extended the results of Posner and Mayne in many different directions (see, for example, \cite{BM, Br1} and references therein). Here, in Section \ref{thms}, we will prove that Mayne's theorem also holds for triangular algebras.

The notion of a generalized derivation was introduced by Bre\v{s}ar \cite{Br} in the late 1980s. A decade later, 
Hvala \cite{H} undertook an algebraic study of generalized derivations; around the same time, Nakajima \cite{N}
studied generalized derivations of rings with identity from a categorical point of view. In 2000, Leger and Luks \cite{LL} investigated generalized derivations of Lie algebras. The notion of a generalized Jordan (respectively, Lie) derivation was introduced by Nakajima \cite{N2}; see also Hvala \cite{H2}. Regarding the triangular algebra setting, Ma and Ji \cite{MaJi} investigated generalized Jordan derivations of upper triangular matrix rings,
while Hou and Qi \cite{HQ1, QH} studied generalized Lie and Jordan derivations of nest algebras. Inspired by 
these works, Benkovi\v{c} \cite{Bk} recently explored generalized Lie and Jordan derivations of (abstract) triangular algebras. Nevertheless, nothing has been said yet about generalized derivations of triangular algebras. In this paper, we close this gap by providing a precise description in Section \ref{geneder}.

The paper is organized as follows: we begin by recalling, in Section \ref{preli}, some basic definitions and results of the theory of linear maps of triangular algebras. In Section \ref{sigmacentra} we focus our attention 
on the so-called $\sigma$-centralizing maps; while Section \ref{geneder} is devoted to the study of generalized $\sigma$-derivations. In the last section, the theorems of Posner and Mayne for the triangular algebra setting are obtained; a more general question involving generalized derivations is also studied. 


\section{Preliminaries} \label{preli}

Throughout the paper by an algebra we mean a unital associative algebra over a fixed commutative unital ring of scalars $R$. In the subsequent subsections, we recall some definitions and basic results, and introduce some notation.

\subsection{Triangular algebras}

Let $A$ and $B$ be algebras and $M$ a nonzero $(A, B)$-bimodule. The following set becomes an associative algebra under the usual matrix operations.
\[
\T = \Tri(A, M, B) = \left(
\begin{array}{cc}
A & M  \\
  & B 
\end{array}
\right) = 
\left\{
\left(
\begin{array}{cc}
a & m  \\
  & b 
\end{array}
\right): \, a \in A, m\in M, b \in B 
\right\}.
\]
An algebra $\T$ is called a {\bf triangular algebra} if there exist algebras $A$, $B$ and a nonzero $(A, B)$-bimodule $M$ such that $\T$ is isomorphic to $\Tri(A, M, B)$. 

Given $\T = \Tri(A, M, B)$ a triangular algebra; let us denote by $\pi_A$, $\pi_B$ the 
two natural projections from $\T$ onto $A$, $B$, respectively defined as follows: 
\[
\pi_A: 
\left(
\begin{array}{cc}
a & m  \\
  & b 
\end{array}
\right) \mapsto a, 
\quad \quad 
\pi_B: 
\left(
\begin{array}{cc}
a & m  \\
  & b 
\end{array}
\right) \mapsto b, \quad  \mbox{ for all } \, \left(
\begin{array}{cc}
a & m  \\
  & b 
\end{array}
\right) \in \T.
\]
The center $\Z(\T)$ of $\T$ was computed in \cite{Ch1} (see also \cite[Proposition 3]{Ch2}); it is the following set:
\[
\Z(\T) = \left\{
\left(
\begin{array}{cc}
a & 0  \\
  & b 
\end{array}
\right) \in \T: \, am = mb, \, \,  \mbox{ for all } m\in M
\right\}.
\]
Moreover, it follows that $\pi_A(\Z(\T)) \subseteq \Z(A)$ and $\pi_B(\Z(\T)) \subseteq \Z(B)$, and there exists a unique algebra isomorphism $\tau:\pi_A(\Z(\T)) \to \pi_B(\Z(\T))$ such that $am = m\tau(a)$, for all $m \in M$.

\noindent The most important examples of triangular algebras are the following:

\smallskip 

\noindent $\bullet$ {\bf Upper triangular matrix algebras.} Let us denote by $\M_{n \times m}(R)$ the algebra of all $n \times m$ matrices over $R$, and by $\T_n(R)$ the algebra of all $n\times n$ upper triangular matrices over $R$. Given $n \geq 2$, the algebra $\T_n(R)$ can be represented as a triangular algebra as follows
\[
\T_n(R) = \left(
\begin{array}{cc}
\T_{\ell}(R) & \quad \M_{\ell \times (n - \ell)}(R) 
\\
&
\\
  & \quad \T_{n - \ell}(R)
\end{array}
\right),
\]
where $\ell \in \{1, \ldots, n - 1\}$.

\smallskip

\noindent $\bullet$ {\bf Block upper triangular matrix algebras.} Let $\NN$ be the set of all positive integers and $n \in \NN$. For any $m\in \NN$ such that $m \leq n$, we write $\bar{d}_n$ to denote an element $(d_1, \ldots, d_m) \in \NN^m$ which satisfies $n = d_1 + \ldots + d_m$. The {\bf block upper triangular matrix algebra} $\B_n^{\bar d_n}(R)$ is the following subalgebra of $\M_n(R)$:
\[
\left(
\begin{array}{ccccc}
\M_{d_1}(R) & \quad \M_{d_1 \times d_2}(R) & \ldots & \M_{d_1 \times d_m}(R) 
\\
&
\\
0 & \quad \M_{d_2}(R) & \ldots & \M_{d_2 \times d_m}(R)
\\
\vdots & \quad \vdots & \ddots & \vdots
\\
0 & \quad 0 & \ldots & \M_{d_m}(R)
\end{array}
\right).
\]
Notice that the full matrix algebra $\M_n(R)$ and the upper triangular matrix algebra $\T_n(R)$
are two special cases of block upper triangular matrix algebras. Given $n\geq 2$, assume that $\B_n^{\bar d_n}(R) \neq \M_n(R)$; then $\B_n^{\bar d_n}(R)$ can be seen as a triangular algebra of the form
\[
B_n^{\bar d_n}(R) = \left(
\begin{array}{cc}
B^{{\bar d}_\ell}_\ell(R) & \quad \M_{\ell \times (n - \ell)}(R) 
\\
&
\\
  & \quad B^{{\bar d}_{n-\ell}}_{n-\ell}(R)
\end{array}
\right),
\]
where $k \in \{1, \ldots, m - 1\}$, 
$\ell = d_{1}+d_{2}+ \dots + d_{k}$,  
${\bar d}_\ell = (d_{1}, d_{2}, \dots, d_{k}) \in \NN^k$,  
and ${\bar d}_{n-\ell} = (d_{k+1}, \dots , d_{m}) \in \NN^{m - k}$.

\smallskip

\noindent $\bullet$ {\bf Triangular Banach algebras.} Let $(A, \| \cdot  \|_A)$ and $(B, \| \cdot  \|_B)$ be two Banach algebras, and $M$ a Banach $(A, B)$-bimodule.
Then the triangular algebra $\T = \Tri(A, M, B)$ is a Banach algebra with respect to the following norm:
\[
\left \| \left(
\begin{array}{cc}
a & m  \\
  & b 
\end{array}
\right) \right \|_\T =  \| a \|_A + \| m \|_M + \| b \|_B, \quad  \mbox{ for all } \, \left(
\begin{array}{cc}
a & m  \\
  & b 
\end{array}
\right) \in \T.
\]
The algebra $\T$ is called a {\bf triangular Banach algebra}.

\smallskip

\noindent $\bullet$ {\bf Nest algebras.}
Let $\mathcal{H}$ be a Hilbert space and $\mathcal{B(H)}$ the algebra of all bounded linear operators on $\mathcal{H}$.
A {\bf nest} is a chain $\mathcal{N}$ of closed subspaces of $\mathcal{H}$, which contains 0 and $\mathcal{H}$, and is closed under arbitrary intersections and closed linear span. A nest $\N$ is said to be {\bf trivial} if $\N = \{0, \Hc \}$; otherwise, it is called a {\bf nontrivial nest}. The {\bf nest algebra} associated to $\N$ is the set 
\[
\Tf(\N) = \{T \in \mathcal{B(H)} \mid T(N) \subseteq N, \, \, \forall \, N \in \N \}.
\]
Any nest algebra $\Tf(\N)$ associated to a nontrivial nest $\N$ can be seen as a triangular algebra. In fact, given $N \in \N - \{0, \Hc\}$, denote by $E$ the orthogonal projection onto $N$. Then $\N_1 = E(\N)$ and $\N_2 = (1 - E)(\N)$ are nests of $N$ and $N^\bot$, respectively. Moreover, $\Tf(\N_1) = E \Tf(\N)E$ and $\Tf(\N_2) = (1 - E) \Tf(\N)(1 - E)$ are nest algebras and 
\[
\Tf(\N) = 
\left(
\begin{array}{ccc}
\Tf(\N_1) & & E \Tf(\N)(1 - E) 
\\
&&
\\
  & & \Tf(\N_2) 
\end{array}
\right).
\]
At this point, it should be mentioned that finite dimensional nest algebras are isomorphic to complex block upper triangular matrix algebras. We refer the reader to \cite{Da} for the general theory of nest algebras. 

\begin{notation}
{\rm 
Let $\T = \Tri(A, M, B)$ be a triangular algebra. 
We will write $1_A$, $1_B$ to denote the units of the algebras $A$, $B$, respectively. The unit of $\T$ is the element: 
\[
1 := \left(
\begin{array}{cl}
1_A & 0  \\
    & 1_B
\end{array}
\right) \in \T. 
\]
It is straightforward to check that the following elements are orthogonal idempotents of $\T$.
\[
p := 
\left(
\begin{array}{cc}
1_A & 0  \\
    & 0
\end{array}
\right), 
\quad \quad
q := 1 - p = 
\left(
\begin{array}{cl}
0 & 0  \\
  & 1_B
\end{array}
\right). 
\]
Hence, we can consider the Peirce decomposition of $\T$ associated to the idempotent $p$. In other words, we can write $\T = p\T p \oplus p \T q \oplus q \T q$. Note that $p\T p$, $q \T q$ are subalgebras of $\T$ isomorphic to $A$, $B$, respectively; while $p \T q$ is a $(p \T p, q \T q)$-bimodule isomorphic to $M$. To ease the notation, we will identify $p\T p$, $q \T q$, $p \T q$ with $A$, $B$ and $M$, respectively. Thus, any element of $\T$ can be expressed as follows:
\[
x = \left(
\begin{array}{cc}
a & m  \\
  & b 
\end{array}
\right) = pap + pmq + qbq = a + m + b,
\]
by using the identification above. 
}
\end{notation}

\subsection{$\sigma$-maps}

\begin{definition}
{\rm
Let $\A$ be an algebra and $\sigma$ an automorphism of $\A$. Let us denote by $\Id_\A$ the identity map on $\A$.
A linear map $d: \A \to \A$ is called a {\bf $\sigma$-derivation of} $\A$ if it satisfies  
\[
d(xy) = d(x)y + \sigma(x)d(y), \quad \forall \, x, y \in \A.
\]
Note that every derivation is an $\Id_\A$-derivation. In the literature (see, for example, \cite{HW, YaZh}) some authors have used the terminology of skew-derivations to denote $\sigma$-derivations. } 
\end{definition}

Given an algebra $\A$ and $\sigma$ an automorphism of $\A$, we introduce the following bilinear maps: 
\begin{equation*} \label{operation}
[x, y]_\sigma = \sigma(x)y - yx, \quad \quad \langle x, y\rangle_\sigma = \sigma(x)y + yx,
\end{equation*}
for all $x, y \in \A$.

\begin{defis}
{\rm
Let $\A$ be an algebra, $\sigma$ an automorphism of $\A$, and a linear map $\Theta: \A \to \A$. The map $\Theta$ is called 
\begin{itemize}
\item {\bf $\sigma$-commuting} (respectively, {\bf $\sigma$-centralizing}) if it satisfies 
$[x, \Theta(x)]_\sigma = 0$ (respectively, $[x, \Theta(x)]_\sigma \in \Z(\A)$), for all $x \in \A$.  
\item {\bf $\sigma$-skew-commuting} (respectively, {\bf $\sigma$-skew-centralizing}) if it verifies 

\noindent $\langle x, \Theta(x)\rangle_\sigma = 0$ (respectively, $\langle x, \Theta(x)\rangle_\sigma \in \Z(\A)$), for all $x \in \A$. 
\end{itemize}
In particular, for $\sigma = \Id_\A$ the notions of commuting (respectively, centralizing), skew-commuting (respectively, skew-centralizing) are recovered. 
}
\end{defis}

\begin{example}
Let $K$ be a field of characteristic not 2 and take 
\[
\A := \left \{ 
\left(
\begin{array}{ccc}
0 & a & b  
\\
0 & 0 & c 
\\
0 & 0 & 0
\end{array}
\right): \, a, b, c \in K 
\right \}.
\]
It is straightforward to show that the map $\sigma: \A \to \A$ given by  
\[
\sigma \left(
\begin{array}{ccc}
0 & a & b  
\\
0 & 0 & c 
\\
0 & 0 & 0
\end{array}
\right) := 
\left(
\begin{array}{ccc}
0 & -a & b  
\\
0 & 0 & -c 
\\
0 & 0 & 0
\end{array}
\right), \quad \mbox{for all} \, \, \left(
\begin{array}{ccc}
0 & a & b  
\\
0 & 0 & c 
\\
0 & 0 & 0
\end{array}
\right) \in \A,
\]  
is an automorphism of $\A$. It is easy to show that the map $\Theta: \A \to \A$ given by  
\[
\Theta \left(
\begin{array}{ccc}
0 & a & b  
\\
0 & 0 & c 
\\
0 & 0 & 0
\end{array}
\right) := 
\left(
\begin{array}{ccc}
0 & a & 0  
\\
0 & 0 & c 
\\
0 & 0 & 0
\end{array}
\right), \quad \mbox{for all} \, \, \left(
\begin{array}{ccc}
0 & a & b  
\\
0 & 0 & c 
\\
0 & 0 & 0
\end{array}
\right) \in \A,
\]  
is $\sigma$-skew-commuting, although it is not skew-commuting. 
For example, for $x = \left(
\begin{array}{ccc}
0 & 1 & 0  
\\
0 & 0 & 1 
\\
0 & 0 & 0
\end{array}
\right)$ we have that $\langle \Theta(x), x \rangle = \left(
\begin{array}{ccc}
0 & 0 & 2  
\\
0 & 0 & 0 
\\
0 & 0 & 0
\end{array}
\right) \neq 0$. 
\end{example}
 
See \cite[Example 2.5]{RSO} for an example of a $\sigma$-commuting map which is not commuting.
 
\subsection{Some basic results on triangular algebras}
Let $\T = \Tri(A, M, B)$ be a triangular algebra. In the sequel (for convenience) we will assume that the bimodule $M$ is faithful as a left $A$-module and also as a right $B$-module, although these assumptions might not always be necessary. 

In what follows, we will use the following result without further mention. 

\begin{theorem} \cite[Theorem 1]{KDW} \label{aut}
Let $\T = \Tri(A, M, B)$ be a triangular algebra such that the algebras $A$ and $B$ have only trivial idempotents. Then every automorphism $\sigma$ of $\T$ is of the following form:
\begin{equation*} \label{autom}
\sigma \left(
\begin{array}{cc}
a & m  \\
  & b 
\end{array}
\right) = 
\left(
\begin{array}{ccc}
f_\sigma(a) && f_\sigma(a)m_\sigma - m_\sigma g_\sigma(b) + \nu_\sigma(m)  
\\
&&
\\
  && g_\sigma(b) 
\end{array}
\right),
\end{equation*}
where $f_\sigma$, $g_\sigma$ are automorphisms of $A$, $B$, respectively, $m_\sigma$ is an element of $M$, and $\nu_\sigma: M \to M$ is a linear bijective map which satisfies $\nu_\sigma(am) = f_\sigma(a)\nu_\sigma(m)$ and $\nu_\sigma(mb) = \nu_\sigma(m) g_\sigma(b)$, for all $a \in A$, $b \in B$, $m \in M$. 
\end{theorem}

A description of derivations of triangular algebras was first given by Forrest and Marcoux \cite{FM} in 1996. Cheung \cite{Ch1} continued, in his thesis, the study of derivations of triangular algebras. Recently, Han and Wei \cite{HW} have generalized Cheung's result by describing $\sigma$-derivations of triangular algebras; they have proved the following result:

\begin{theorem} \cite[Theorem 3.12]{HW} \label{sigmader}
Let $\T = \Tri(A, M, B)$ be a triangular algebra and $\sigma$ an automorphism of $\T$. Assume that the algebras $A$ and $B$ have only trivial idempotents. Then every $\sigma$-derivation $d$ of $\T$ is of the following form: 
\begin{equation} \label{sgder} 
d \left(
\begin{array}{cc}
a & m  \\
  & b 
\end{array}
\right) = 
\left(
\begin{array}{ccc}
d_A(a) && f_\sigma(a)m_d - m_d b - m_\sigma d_B(b) + \xi(m)  
\\
&&
\\
  && d_B(b) 
\end{array}
\right),
\end{equation}
where $d_A$ is an $f_\sigma$-derivation of $A$, $d_B$ is a $g_\sigma$-derivation of $B$, $m_d$ is an element of $M$, and $\xi: M \to M$ is a linear map which satisfies 
\begin{equation*} 
\xi(am) = d_A(a)m + f_\sigma(a)\xi(m), \quad \quad \quad \xi(mb) = \xi(m)b + \nu_\sigma(m) d_B(b),
\end{equation*} 
for all $a \in A$, $b \in B$, $m \in M$.
\end{theorem}

\begin{corollary} \cite{FM} \label{descrider} 
Every derivation of a triangular algebra $\T = \Tri(A, M, B)$ can be written as follows: 
\begin{equation} \label{deriv} 
d \left(
\begin{array}{cc}
a & m  \\
  & b 
\end{array}
\right) = 
\left(
\begin{array}{ccc}
d_A(a) && am_d - m_d b + \xi(m)  
\\
&&
\\
  && d_B(b) 
\end{array}
\right),
\end{equation}
where $d_A$ is a derivation of $A$, $d_B$ is a derivation of $B$, $m_d$ is an element of $M$, and $\xi: M \to M$ is a linear map which satisfies 
\begin{equation*} 
\xi(am) = d_A(a)m + a\xi(m), \quad \quad \quad \xi(mb) = \xi(m)b + m d_B(b),
\end{equation*} 
for all $a \in A$, $b \in B$, $m \in M$.
\end{corollary}

\begin{remark}
In the paper, we will mostly consider triangular algebras $\T = \Tri(A, M, B)$ for which 
$A$ and $B$ have only trivial idempotents. The main reason is that the description of automorphisms of $\T$, provided in Theorem \ref{aut}, will be used in the sequel. 

We refer the reader to \cite[Section 2.7]{HW} for examples of triangular algebras with $A$ and $B$ having only trivial idempotents. 
\end{remark}

The following concepts and results will play an important role for our purposes.

\begin{definition} \cite[Definition 2.3]{YaZh}
Let $\A$ be an algebra and $\sigma$ an automorphism of $\A$. The {\bf $\sigma$-center} of $\A$ is the set $\Z_\sigma (\A)$ given by 
\begin{align*}
\Z_\sigma (\A) & = \{ \lambda \in \A : \sigma(x)\lambda = \lambda x, \, \, \mbox{ for all } \, x \in \A  \}
\\
& = \{ \lambda \in \A : [x, \lambda]_\sigma = 0, \, \, \mbox{ for all } \, x \in \A  \}.
\end{align*}
A linear map from $\A$ to $\Z_\sigma (\A)$ will be called {\bf $\sigma$-central}.
\end{definition}
 
\begin{proposition} \cite[Lemma 2.5]{YaZh}
Let $\A$ be an algebra and $\sigma$ an automorphism of $\A$. The $\sigma$-center of $\A$ is a subspace of $\A$, which is closed under applications of $\sigma$; in other words, $\sigma(\lambda) \in \Z_\sigma (\A)$ for every $\lambda \in  \Z_\sigma (\A)$.
\end{proposition}

\begin{proposition} \cite[Lemmas 2.6, 2.8 and 2.9]{YaZh} \label{center}
Let $\T = \Tri(A, M, B)$ be a triangular algebra and $\sigma$ an automorphism of $\T$. Assume that the algebras $A$ and $B$ have only trivial idempotents. Then 
\[
\Z_\sigma (\T) =
\left\{
\left(
\begin{array}{cc}
a & - m_\sigma b  \\
  &       b 
\end{array}
\right) \in \T \, \mbox{ such that } \, am = \nu_\sigma(m)b, \, \mbox{ for all } \, m \in M 
\right\}.
\]
Moreover, $\pi_A (\Z_\sigma (\T)) \subseteq \Z_{f_\sigma} (A)$, $\pi_B (\Z_\sigma (\T)) \subseteq \Z_{g_\sigma} (B)$, and there exists a unique algebra isomorphism $\eta: \pi_B (\Z_\sigma (\T)) \to \pi_A (\Z_\sigma (\T))$ such that $\eta(b)m = \nu_\sigma(m)b$ for all $b\in \Z_\sigma (\T)$, $m \in M$. 
\end{proposition}

In particular for $\sigma = \Id_A$ we get the following result, which was originally proved by Cheung.

\begin{corollary} \cite[Proposition 3]{Ch2} \label{center2}
Let $\T = \Tri(A, M, B)$ be a triangular algebra. Then the center of $\T$ is given by
\[
\Z (\T) =
\left\{
\left(
\begin{array}{cc}
a & 0  \\
  & b 
\end{array}
\right) \in \T, \, \mbox{ such that } \, am = mb, \, \mbox{ for all } \, m \in M 
\right\}.
\]
Moreover, $\pi_A (\Z(\T)) \subseteq \Z (A)$, $\pi_B (\Z (\T)) \subseteq \Z (B)$, and there exists a unique algebra isomorphism $\eta: \pi_B (\Z (\T)) \to \pi_A (\Z (\T))$ such that $\eta(b)m = mb$, for all $b\in \Z(\T)$, $m \in M$. 
\end{corollary}


\section{$\sigma$-centralizing maps of triangular algebras} \label{sigmacentra}

Given an algebra $\A$, $\sigma$ an automorphism of $\A$, and $\Theta$ a $\sigma$-centralizing map of $\A$; we have that 
\begin{equation} \label{ce1}
[x, \Theta(x)]_\sigmaÊ\in \Z(\A), \quad \forall \, x \in \A \tag{$\sigma$-cent1}.
\end{equation}
A linearization of the equation above produces the following:
\begin{equation} \label{ce2}
[x, \Theta(y)]_\sigmaÊ+ [y, \Theta(x)]_\sigma \in \Z(\A), \quad \forall \, x, y \in \A \tag{$\sigma$-cent2}.
\end{equation}


In this section, we investigate $\sigma$-centralizing maps of triangular algebras. A precise description of $\sigma$-centralizing maps is stated in the following theorem.

\begin{theorem} \label{descricent}
Let $\T = \Tri(A, M, B)$ be a triangular algebra and $\sigma$ an automorphism of $\T$. Assume that the algebras $A$ and $B$ have only trivial idempotents. Then every $\sigma$-centralizing map $\Theta$ of $\T$ can be written as follows:
\begin{equation*} 
\tag{$\sigma$-cent} \label{eqdescent}
\Theta 
\left( 
\begin{array}{cc}
a & m  \\
  & b
\end{array}
\right) = 
\left( 
\begin{array}{ccc}
\delta_1(a) + \delta_2(m) + \delta_3(b) && - m_\sigma(\mu_1(a) + \mu_2(m) + \mu_3(b)) \, + \\
                                        && \delta_1(1_A)m - \nu_\sigma(m)\mu_1(1_A) 
  \\
  &&
  \\
  && \mu_1(a) + \mu_2(m) + \mu_3(b)
\end{array}
\right),
\end{equation*}
where 
\begin{alignat*}{3}
& \delta_1: A \to A, \quad & \delta_2: M \to \Z_{f_\sigma}(A), \quad  & \delta_3: B \to A, 
\\
& \mu_1: A \to B, \quad & \mu_2: M \to \Z_{g_\sigma}(B), \quad  & \mu_3: B \to B,
\end{alignat*}
are linear maps which satisfy the following conditions: 
\begin{itemize}
\item[{\rm (i)}] $\delta_1$ is an $f_\sigma$-commuting map of $A$,
\item[{\rm (ii)}] $\mu_3$ is a $g_\sigma$-commuting map of $B$,
\item[{\rm (iii)}] $\delta_1(a)m - \nu_\sigma(m)\mu_1(a) = f_\sigma(a)(\delta_1(1_A)m - \nu_\sigma(m)\mu_1(1_A))$,
\item[{\rm (iv)}] $\nu_\sigma(m)\mu_3(b) - \delta_3(b)m = (\nu_\sigma(m)\mu_3(1_B) - \delta_3(1_B)m )b$,
\item[{\rm (v)}] $\delta_2(m)m = \nu_\sigma(m)\mu_2(m)$,
\item[{\rm (vi)}] $\delta_1(1_A)m - \nu_\sigma(m)\mu_1(1_A) = \nu_\sigma(m)\mu_3(1_B)  - \delta_3(1_B)m$,
\item[{\rm (vii)}] $[a, \delta_3(b)]_{f_\sigma} \in \Z(A)$,
\item[{\rm (viii)}] $[b, \mu_1(a)]_{g_\sigma} \in \Z(B)$,
\end{itemize}
for all $a\in A$, $m\in M$ and $b\in B$. 
\end{theorem}

\begin{proof}
Let $\Theta$ be a $\sigma$-centralizing map of $\T$; we write $\Theta$ as follows:  
\[
\Theta 
\left( 
\begin{array}{cc}
a & m  \\
  & b
\end{array}
\right) = 
\left( 
\begin{array}{ccc}
\delta_1(a) + \delta_2(m) + \delta_3(b) && \tau_1(a) + \tau_2(m) + \tau_3(b)) 
 \\
 &&
 \\
  && \mu_1(a) + \mu_2(m) + \mu_3(b)
\end{array}
\right),
\] 
for all $\left( 
\begin{array}{cc}
a & m  \\
  & b
\end{array}
\right) \in \T$. 
The proof consists of applying \eqref{ce1} and \eqref{ce2} with appropriate elements of $\T$. Given $a \in A$;  let us start by 
applying \eqref{ce1} with $x = \left( 
\begin{array}{cc}
a & 0  \\
  & 0
\end{array}
\right)$ to get that 
\begin{equation} \label{uno0}
[a, \delta_1(a)]_{f_\sigma}M = 0, \quad \quad f_\sigma(a) (\tau_1(a) + m_\sigma\mu_1(a) ) = 0,
\end{equation}
Taking into account that $M$ is a faithful left $A$-module, from the first equation above we obtain that 
$[a, \delta_1(a)]_{f_\sigma} = 0$, for all $a \in A$, which shows (i). On the other hand, making $a = 1_A$ in the second equation of \eqref{uno0} we have that 
\begin{equation} \label{uno1}
\tau_1(1_A) + m_\sigma \mu_1(1_A)= 0.
\end{equation}
Repeating the same process, with the element $x = \left( 
\begin{array}{cc}
0 & 0  \\
  & b
\end{array}
\right) \in \T$, using the fact that $M$ is a faithful right $B$-module, we prove (ii). Also, we obtain that 
\begin{equation} \label{two1}
m_\sigma \mu_3(1_B) + \tau_3(1_B) = 0.
\end{equation} 
Next, from \eqref{ce2} with the elements $x = \left( 
\begin{array}{cc}
a & 0  \\
  & b
\end{array}
\right)$ and $y = q$ we get the following: 
\begin{equation}\label{tres0}
\left( 
\begin{array}{ccc}
[a, \delta_3(1_B)]_{f_\sigma} && -m_\sigma g_\sigma(b)\mu_3(1_B) - \tau_3(1_B)b - m_\sigma\mu_1(a) 
\\
&& - \tau_1(a) - m_\sigma \mu_3(b) - \tau_3(b)  
\\
&&
\\
  && [b, \mu_3(1_B)]_{g_\sigma}
\end{array}
\right) \in \Z(\T),
\end{equation}
for all $a\in A$ and $bÊ\in B$. Making $a = 0$ in \eqref{tres0}, we get that 
\begin{equation*}
\left( 
\begin{array}{cccc}
0 &&& -m_\sigma g_\sigma(b)\mu_3(1_B) - \tau_3(1_B)b 
\\
&&&- m_\sigma \mu_3(b) - \tau_3(b)  
\\
&&&
\\
  &&& [b, \mu_3(1_B)]_{g_\sigma}
\end{array}
\right) \in \Z(\T),
\end{equation*}
for all $b \in B$. This allows us to conclude that 
\begin{equation} \label{tres1}
m_\sigma g_\sigma(b)\mu_3(1_B) + \tau_3(1_B)b + m_\sigma \mu_3(b) + \tau_3(b) = 0, \quad M[b, \mu_3(1_B)]_{g_\sigma} = 0,
\end{equation}
for all $b \in B$. Taking into account the fact that $M$ is a faithful right $B$-module, the second equation in \eqref{tres1} yields that 
\begin{equation} \label{tres2}
[b, \mu_3(1_B)]_{g_\sigma} = 0, \quad \forall \, b \in B.
\end{equation}
Applying \eqref{tres2}, the first equation of \eqref{tres1} becomes:
\[
(m_\sigma \mu_3(1_B) + \tau_3(1_B))b + m_\sigma \mu_3(b) + \tau_3(b) = 0
\]
for all $b \in B$. Now, \eqref{two1} applies to get that
\begin{equation} \label{tres3}
\tau_3(b) = - m_\sigma \mu_3(b), \quad \forall \, b \in B.
\end{equation}
Similarly, making $b = 0$ in \eqref{tres0} allows us to conclude that 
\begin{equation} \label{tres4}
\tau_1(a) = - m_\sigma \mu_1(a), \quad \forall \, a \in A.
\end{equation}
Next, considering $x = \left( 
\begin{array}{cc}
a & 0  \\
  & 0
\end{array}
\right)$ and $y = \left( 
\begin{array}{cc}
0 & m  \\
  & 0
\end{array}
\right)$ in \eqref{ce2} implies that 
\begin{equation*}
\left( 
\begin{array}{cccc}
[a, \delta_2(m)]_{f_\sigma} &&& f_\sigma(a)(\tau_2(m) + m_\sigma \mu_2(m))   
\\
&&& + \nu_\sigma(m) \mu_1(a) - \delta_1(a)m  
\\
&&&
\\
  &&& 0
\end{array}
\right) \in \Z(\T),
\end{equation*}
for all $a \in A$ and $m \in M$; which yields that 
\begin{align} \label{cuatro0}
& [a, \delta_2(m)]_{f_\sigma} M = 0, 
\\ \label{cuatro1}
& f_\sigma(a)(\tau_2(m) + m_\sigma \mu_2(m)) = \delta_1(a)m - \nu_\sigma(m) \mu_1(a),
\end{align}  
for all $a \in A$ and $m \in M$. Applying the fact that $M$ is a faithful left $A$-module in \eqref{cuatro0}
we get that $\delta_2(M) \subseteq \Z_{f_\sigma}(A)$. On the other hand, making $a = 1_A$ in \eqref{cuatro1} gives
\begin{equation} \label{cuatro2}
\tau_2(m) + m_\sigma \mu_2(m) = \delta_1(1_A)m - \nu_\sigma(m) \mu_1(1_A), \quad \forall \, m \in M. 
\end{equation} 
From \eqref{cuatro1} and \eqref{cuatro2} we prove (iii). Taking $x = \left( 
\begin{array}{cc}
0 & 0  \\
  & b
\end{array}
\right)$ and $y = \left( 
\begin{array}{cc}
0 & m  \\
  & 0
\end{array}
\right)$ in \eqref{ce2} produces that
\begin{equation*}
\left( 
\begin{array}{cccc}
0 &&& - m_\sigma g_\sigma(b)\mu_2(m) -\tau_2(m)b   
\\
&&& + \nu_\sigma(m) \mu_3(b) - \delta_3(b)m  
\\
&&&
\\
  &&& [b, \mu_2(m)]_{g_\sigma}
\end{array}
\right) \in \Z(\T),
\end{equation*}
for all $m \in M$ and $b \in B$. From this, we derive that 
\begin{align} 
& M[b, \mu_2(m)]_{g_\sigma} = 0, \label{cinco0}
\\ \label{cinco1}
&  - m_\sigma g_\sigma (b) \mu_2(m) -\tau_2(m)b +  \nu_\sigma(m) \mu_3(b) - \delta_3(b)m = 0,
\end{align}
for all $m \in M$ and $b \in B$. Equation \eqref{cinco0} gives that $\mu_2(M) \subseteq \Z_{g_\sigma}(B)$, 
applying the fact that $M$ is a faithful right $B$-module. From here, \eqref{cinco1} becomes 
\begin{equation} \label{cinco2}
\nu_\sigma(m) \mu_3(b) - \delta_3(b)m = (m_\sigma \mu_2(m) + \tau_2(m)) b,
\end{equation}
for all $m \in M$ and $b \in B$. In particular, for $b = 1_B$ we get that 
\begin{equation} \label{cinco3}
\nu_\sigma(m) \mu_3(1_B) - \delta_3(1_B)m = m_\sigma \mu_2(m) + \tau_2(m), \quad \forall \, m \in M.
\end{equation}
From \eqref{cinco2} and \eqref{cinco3} we show (iv). Moreover, (vi) follows from \eqref{cuatro2} and \eqref{cinco3}. Given $m\in M$, an application of \eqref{ce1} with the element $x = \left( 
\begin{array}{cc}
0 & m  \\
  & 0
\end{array}
\right)$ shows (v). On the other hand, applications of \eqref{tres3}, \eqref{tres4} and \eqref{cuatro2} 
allow us to conclude that 
\[
\tau_1(a) + \tau_2(m) + \tau_3(b) = -m_\sigma (\mu_1(a) + \mu_2(m) + \mu_3(b)) + \delta_1(1_A)m - \nu_\sigma(m) \mu_1(1_A),
\] 
as desired. It remains to show (vii) and (viii). To this end, take $a \in A$, $b \in B$ and apply \eqref{ce2} with the elements $x = \left( 
\begin{array}{cc}
a & 0  \\
  & 0
\end{array}
\right)$ and $y = \left( 
\begin{array}{cc}
0 & 0 \\
  & b
\end{array}
\right)$ to obtain that 
\begin{equation*}
\left( 
\begin{array}{cccc}
[a, \delta_3(b)]_{f_\sigma} &&& - m_\sigma g_\sigma(b)\mu_1(a) -\tau_1(a)b   
\\
&&&
\\
  &&& [b, \mu_1(a)]_{g_\sigma}
\end{array}
\right) \in \Z(\T);
\end{equation*}
which jointly with Corollary \ref{center2} implies that 
\[
[a, \delta_3(b)]_{f_\sigma} \in \pi_A (\Z(\T)) \subseteq \Z(A), \quad \quad [b, \mu_1(a)]_{g_\sigma} \in \pi_B (\Z(\T)) \subseteq \Z(B),
\]
finishing the proof.
\end{proof}

As a consequence of Theorem \ref{descricent}, 
the description of centralizing maps of triangular algebras is obtained. 

\begin{corollary} \label{centra}
Every centralizing map of a triangular algebra $\T = \Tri(A, M, B)$ has the following form:
\begin{equation*} 
\Theta 
\left( 
\begin{array}{cc}
a & m  \\
  & b
\end{array}
\right) = 
\left( 
\begin{array}{ccc}
\delta_1(a) + \delta_2(m) + \delta_3(b) &&  \delta_1(1_A)m - m\mu_1(1_A) 
  \\
  &&
  \\
  && \mu_1(a) + \mu_2(m) + \mu_3(b)
\end{array}
\right),
\end{equation*}
where 
\begin{alignat*}{3}
& \delta_1: A \to A, \quad & \delta_2: M \to \Z(A), \quad  & \delta_3: B \to A, 
\\
& \mu_1: A \to B, \quad & \mu_2: M \to \Z(B), \quad  & \mu_3: B \to B,
\end{alignat*}
are linear maps such that 
\begin{itemize}
\item[{\rm (i)}] $\delta_1$ is a commuting map of $A$,
\item[{\rm (ii)}] $\mu_3$ is a commuting map of $B$,
\item[{\rm (iii)}] $\delta_1(a)m - m\mu_1(a) = a(\delta_1(1_A)m - m\mu_1(1_A))$,
\item[{\rm (iv)}] $m\mu_3(b) - \delta_3(b)m = (m\mu_3(1_B) - \delta_3(1_B)m )b$,
\item[{\rm (v)}] $\delta_2(m)m = m\mu_2(m)$,
\item[{\rm (vi)}] $\delta_1(1_A)m - m\mu_1(1_A) = m\mu_3(1_B)  - \delta_3(1_B)m$,
\item[{\rm (vii)}] $[a, \delta_3(b)] \in \Z(A)$,
\item[{\rm (viii)}] $[b, \mu_1(a)] \in \Z(B)$,
\end{itemize}
for all $a\in A$, $m\in M$ and $b\in B$. 
\end{corollary}

The $\sigma$-commuting maps of triangular algebras have been recently explored by Repka and the author \cite{RSO}. Their description (\cite[Theorem 4.2]{RSO}) can be obtained now as a consequence of Theorem \ref{descricent}.
Moreover, \cite[Theorem 4.2]{RSO} and Theorem \ref{descricent} give us necessary and sufficient conditions for a $\sigma$-centralizing map of a triangular algebra to be $\sigma$-commuting. They have been collected in the two following corollaries.

\begin{corollary} 
Let $\T = \Tri(A, M, B)$ be a triangular algebra and $\sigma$ an automorphism of $\T$. Assume that the algebras $A$ and $B$ have only trivial idempotents. A $\sigma$-centralizing map $\Theta$ of $\T$ is $\sigma$-commuting if and only if $\delta_3(B) \subseteq \Z_{f_\sigma}(A)$ and $\mu_1(A) \subseteq \Z_{g_\sigma}(B)$.
\end{corollary}

\begin{corollary} \label{condicomm}
Let $\T = \Tri(A, M, B)$ be a triangular algebra. A centralizing map $\Theta$ of $\T$ is commuting if and only if $\delta_3(B) \subseteq \Z(A)$ and $\mu_1(A) \subseteq Z(B)$.
\end{corollary}

In 1993, Bre\v{s}ar \cite{Br2} proved that zero is the only skew-commuting map on a 2-torsion-free (i.e., $2x \neq 0$ for every nonzero element $x$) semiprime ring. We close this section by analyzing what happens with the $\sigma$-skew-commuting maps of 2-torsion-free triangular algebras.

\begin{theorem}
Let $\T = \Tri(A, M, B)$ be a triangular algebra and $\sigma$ an automorphism of $\T$. Assume that $\T$ is 2-torsion-free and that $A$ and $B$ have only trivial idempotents. Then zero is the only $\sigma$-skew-commuting map of $\T$.
\end{theorem}

\begin{proof}
Let $\Theta_s$ be a skew-commuting map of $\T$. Then we have that
\begin{equation} \label{scom1}
\langle x, \Theta_s(x) \rangle_\sigmaÊ= 0, \quad \forall \, x \in \T \tag{$\sigma$-scom1}.
\end{equation}
Linearizing \eqref{scom1} we obtain that 
\begin{equation} \label{scom2}
\langle x, \Theta(y)\rangle_\sigmaÊ+ \langle y, \Theta(x) \rangle_\sigma = 0,  \quad \forall \, x, y \in \T \tag{$\sigma$-scom2}.
\end{equation}
In what follows, we will apply \eqref{scom1} and \eqref{scom2} with 
carefully chosen elements of $\T$. 
Assume that 
\[
\Theta_s 
\left( 
\begin{array}{cc}
a & m  \\
  & b
\end{array}
\right) = 
\left( 
\begin{array}{ccc}
\delta_1(a) + \delta_2(m) + \delta_3(b) && \tau_1(a) + \tau_2(m) + \tau_3(b)) 
 \\
 &&
 \\
  && \mu_1(a) + \mu_2(m) + \mu_3(b)
\end{array}
\right),
\] 
for all $\left( 
\begin{array}{cc}
a & m  \\
  & b
\end{array}
\right) \in \T$. Let us start by applying \eqref{scom1} with $x = p$; it gives that 
\[
\tau_1(1_A) + m_\sigma \mu_1(1_A) = 0, \quad \quad  2\delta_1(1_A) = 0,
\]
which yields that 
\begin{equation} \label{eqskew1}
\tau_1(1_A) + m_\sigma \mu_1(1_A) = 0, \quad \quad  \delta_1(1_A) = 0,
\end{equation}
since $\T$ is 2-torsion-free. Using \eqref{scom1} with $x = q$ and applying the fact 
that $\T$ is 2-torsion-free we obtain that
\begin{equation} \label{eqskew2}
\mu_3(1_B) = 0, \quad \quad \quad \tau_3(1_B) = 0.
\end{equation} 
Next, take $x = \left( 
\begin{array}{cc}
a & 0  \\
  & b
\end{array}
\right)$ and $y = p$ and apply \eqref{scom2} jointly with \eqref{eqskew1} to obtain 
\begin{align} \label{eqskew3}
2\delta_1(a) + 2\delta_3(b) & = 0,
\\ \notag
\langle b, \mu_1(1_A) \rangle_{g_\sigma} & = 0, 
\\ \notag
\tau_1(a) + m_\sigma \mu_1(a) + \tau_3(b) + m_\sigma \mu_3(b) & = 0.
\end{align}
Letting
$b = 0$ in \eqref{eqskew3} and using 
the fact that $\T$ is 2-torsion-free 
gives: 
\begin{equation} \label{eqskew4}
\tau_1(a) + m_\sigma \mu_1(a) = 0, \quad \quad  \delta_1(a) = 0, \quad \, \forall \, a \in A.
\end{equation}
Taking $a = 0$ in \eqref{eqskew3} and applying
the fact that $\T$ is 2-torsion-free allow us to conclude that 
\begin{equation} \label{eqskew5}
\tau_3(b) + m_\sigma \mu_3(b) = 0, \quad \quad  \delta_3(b) = 0, \quad \, \forall \, b \in B.
\end{equation}
Applying \eqref{scom2} with $x = \left( 
\begin{array}{cc}
a & 0  \\
  & b
\end{array}
\right)$ and $y = q$, jointly with \eqref{eqskew2}, shows that 
\begin{equation} \label{eqskew6}
2\tau_1(a) + 2\tau_3(b) = 0, \quad \quad 2\mu_1(a) + 2\mu_3(b) = 0, \quad \, \forall \, a \in A, \, b \in B.
\end{equation}
Making first $a = 0$ in \eqref{eqskew6}, and secondly $b = 0$ establishes that 
\begin{equation} \label{eqskew7}
\tau_1(a) = \mu_1(a) = 0, \quad \quad \tau_3(b) = \mu_3(b) = 0, \quad \, \forall \, a \in A, \, b \in B.
\end{equation}
It remains to show that $\delta_2 = \tau_2 = \mu_2 = 0$. To this end, apply \eqref{eqskew2} with the elements
$x = \left( 
\begin{array}{cc}
0 & m  \\
  & 0
\end{array}
\right)$ and 
$y = \left( 
\begin{array}{cc}
a & 0  \\
  & b
\end{array}
\right)$; which implies that 
\begin{align} \label{eqskew8}
\langle a, \delta_2(m) \rangle_{f_\sigma} & = 0,
\\ \notag
\langle b, \mu_2(m) \rangle_{g_\sigma} & = 0,
\\ \notag
f_{\sigma}(a)(\tau_2(m) + m_\sigma \mu_2(m)) + (\tau_2(m) + m_\sigma \mu_2(m))b & = 0.
\end{align}
Making $a = 1_A$ and $b = 1_B$ in \eqref{eqskew8} gives the following:
\[
2\delta_2(m) = 2\mu_2(m) = 0, \quad \quad 2(\tau_2(m) + m_\sigma \mu_2(m)) = 0.
\]
Taking into account the fact that $\T$ is 2-torsion-free, we obtain 
\[
\delta_2(m) = \mu_2(m) = \tau_2(m) = 0, \quad \, \forall \, m \in M,
\]
as desired. 
\end{proof}

\begin{corollary}
Zero is the only skew-commuting map of a 2-torsion-free triangular algebra. 
\end{corollary}

At this point, one may ask what can be said about skew-centralizing maps of triangular algebras. We 
can apply the following result due to Sharma and Dhara:

\begin{theorem} \cite{SD}
Let $\R$ be an arbitrary 2-torsion-free ring with left identity. Then any skew-centralizing map of $\R$ degenerates to a commuting map. 
\end{theorem}


\section{Generalized $\sigma$-derivations of triangular algebras} \label{geneder}

Let $\A$ be an algebra and $a, b \in \A$, a {\bf generalized inner derivation} $\delta_{a, b}$ of $\A$ is a map of the form $\delta_{a, b}(x) =  ax + xb$, for $x \in \A$. In the theory of operator algebras, generalized inner derivations constitute an important class of the so-called elementary operators,   
which are maps of the form $x \mapsto \sum^n_{i = 1} a_ixb_i$. (See \cite{Ma} for an account of this theory.) Note that every generalized inner derivation $\delta_{a, b}$ of $\A$ can be written as follows:
\[
\delta_{a, b}(xy) = \delta_{a, b}(x)y + x\delta_b(y), \quad \, \forall \, \, x, y \in \A,
\]
where $\delta_b(x) = bx - xb$, for all $x \in \A$, is an inner derivation. Inspired by this fact, Bre\v{s}ar \cite{Br} introduced the following notion: an additive map $D$ of $\A$ is called a {\bf generalized derivation of} $\A$ if there exists a derivation $d$ of $\A$ such that 
\[
D(xy) = D(x)y + xd(y), \quad \, \forall \, \, x, y \in \A.
\]
In general, the derivation $d$ associated to a generalized derivation $D$ is not mentioned; although, sometimes we will refer to $d$ as the {\bf associated derivation of} $D$. Notice that the concept of generalized derivations covers both the notions of derivations and generalized inner derivations. Moreover, the throughly studied {\bf left multipliers}: additive maps $F$ satisfying $F(xy) = F(x)y$, for all $x, y$, are a particular case of generalized derivations.

Let $\sigma$ be an automorphism of $\A$. A linear map $D: \A \to \A$ is said to be a {\bf generalized $\sigma$-derivation of} $\A$ if there exists a $\sigma$-derivation $d$ of $\A$ such that
\begin{equation} \label{gender}
D(xy) = D(x)y + \sigma(x)d(y),
\end{equation}
for all $x, y \in \A$.

\begin{example}
Let $\T = \Tri(A, A, 0)$ and $\sigma_A$ an automorphism of $A$. It is straightforward to show that $\sigma: \A \to \A$ given by 
\[
\sigma \left( 
\begin{array}{cc}
a_{11} & a_{12}  \\
  & 0
\end{array}
\right) = 
\left( 
\begin{array}{ccc}
\sigma_A(a_{11}) && \sigma_A(a_{12})  \\
&&
\\
  && 0
\end{array}
\right), \quad \mbox{for all } \, \,  
\left( 
\begin{array}{cc}
a_{11} & a_{12}  \\
  & 0
\end{array}
\right) \in \T,
\] 
is an automorphism of $\T$. It is easy to check that the map $d: \T \to \T$ defined as 
\[
d \left( 
\begin{array}{cc}
a_{11} & a_{12}  \\
  & 0
\end{array}
\right) = 
\left( 
\begin{array}{ccc}
0 && \sigma_A(a_{11})  \\
&&
\\
  && 0
\end{array}
\right), \quad \mbox{for all } \, \,  
\left( 
\begin{array}{cc}
a_{11} & a_{12}  \\
  & 0
\end{array}
\right) \in \T,
\] 
is a $\sigma$-derivation of $\T$; while the map $D: \T \to \T$ given by 
\[
D \left( 
\begin{array}{cc}
a_{11} & a_{12}  \\
  & 0
\end{array}
\right) = 
\left( 
\begin{array}{ccc}
a_{11} && \sigma_A(a_{11}) + a_{12} \\
&&
\\
  && 0
\end{array}
\right), \quad \mbox{for all } \, \,  
\left( 
\begin{array}{cc}
a_{11} & a_{12}  \\
  & 0
\end{array}
\right) \in \T,
\]
is a generalized $\sigma$-derivation of $\T$ with associated $\sigma$-derivation $d$. For $\sigma = \Id_\T$ we have \cite[Example 1]{A}.

Note that $d$ and $D$ are not in general a derivation and a generalized derivation of $\T$, respectively. For example, let $K$ be a field of characteristic not 2, and take $A = K[x]$ and $\sigma_A$ the automorphism of $A$
which sends $x$ to $-x$. For $\T$, $d$ and $D$ as above, and $a = b = \left( 
\begin{array}{cc}
x & x  \\
  & 0
\end{array}
\right)$, we have that $ab = \left( 
\begin{array}{cc}
x^2 & x^2  \\
  & 0
\end{array}
\right)$, and 
\allowdisplaybreaks 
\begin{align*}
d(ab) & = \left( 
\begin{array}{ccc}
0 && \sigma_A(x^2)  \\
&&
\\
  && 0
\end{array}
\right) = \left( 
\begin{array}{ccc}
0 && \sigma_A(x)\sigma_A(x)  \\
&&
\\
  && 0
\end{array}
\right) = \left( 
\begin{array}{ccc}
0 && x^2 \\
&&
\\
  && 0
\end{array}
\right),
\\
d(a)b + a d(b) & = \left( 
\begin{array}{ccc}
0 && -x \\
&&
\\
  && 0
\end{array}
\right) \left( 
\begin{array}{ccc}
x && x \\
&&
\\
  && 0
\end{array}
\right) + \left( 
\begin{array}{ccc}
x && x \\
&&
\\
  && 0
\end{array}
\right)\left( 
\begin{array}{ccc}
0 && -x \\
&&
\\
  && 0
\end{array}
\right) 
\\
& = \left( 
\begin{array}{ccc}
0 && -x^2 \\
&&
\\
  && 0
\end{array}
\right),
\end{align*}
which shows that $d$ is not a derivation of $\T$. On the other hand, we claim that $D$ is not a generalized derivation of $\T$. In fact, assume on the contrary that $D$ is a generalized derivation of $\T$. Then we can find a derivation $\widetilde{d}$ of $\T$ such that $D(ab) = D(a)b + a \widetilde{d}(b)$, for all $a, b \in \T$. Take $a = \left( 
\begin{array}{cc}
x & x  \\
  & 0
\end{array}
\right)$, $b = \left( 
\begin{array}{cc}
x & 0  \\
  & 0
\end{array}
\right)$, and assume that $\widetilde{d}(b) = \left( 
\begin{array}{cc}
b_{11} & b_{12} \\
  & 0
\end{array}
\right)$. Then we have that   
\begin{align*}
D(ab) & = D\left( 
\begin{array}{ccc}
x^2 && 0  \\
&&
\\
  && 0
\end{array}
\right) = \left( 
\begin{array}{ccc}
x^2 && -x^2  \\
&&
\\
  && 0
\end{array}
\right),
\\
D(a)b + a d(b) & = \left( 
\begin{array}{ccc}
x && 0  \\
&&
\\
  && 0
\end{array}
\right) 
\left( 
\begin{array}{ccc}
x && 0  \\
&&
\\
  && 0
\end{array}
\right) + 
\left( 
\begin{array}{ccc}
x && x  \\
&&
\\
  && 0
\end{array}
\right) 
\left( 
\begin{array}{ccc}
b_{11} && b_{12}   \\
&&
\\
  && 0
\end{array}
\right) 
\\
& = 
\left( 
\begin{array}{ccc}
x^2 + xb_{11} && xb_{12}   \\
&&
\\
  && 0
\end{array}
\right),
\end{align*}
which implies that $b_{11} = 0$ and $b_{12} = -x$. Thus: $\widetilde{d}(b) = \left( 
\begin{array}{cc}
0 & -x \\
  & 0
\end{array}
\right)$. For $a = b = \left( 
\begin{array}{cc}
x & 0  \\
  & 0
\end{array}
\right)$ we get that
\begin{align*}
D(ab) & = \left( 
\begin{array}{ccc}
x^2 && x^2  \\
&&
\\
  && 0
\end{array}
\right),
\\
D(a)b + a d(b) & = \left( 
\begin{array}{ccc}
x && -x  \\
&&
\\
  && 0
\end{array}
\right)
\left( 
\begin{array}{ccc}
x && 0  \\
&&
\\
  && 0
\end{array}
\right) + 
\left( 
\begin{array}{ccc}
x && 0 \\
&&
\\
  && 0
\end{array}
\right)
\left( 
\begin{array}{ccc}
0 && -x  \\
&&
\\
  && 0
\end{array}
\right)
\\
& = \left( 
\begin{array}{ccc}
x^2 && -x^2  \\
&&
\\
  && 0
\end{array}
\right),
\end{align*} 
a contradiction.
\end{example}

See, for example, \cite{DP} for a 
study of generalized $(\sigma, \tau)$-derivations of semiprime rings. 
In this paper, we will restrict our attention to the triangular algebra setting. Our first aim will be to describe generalized $\sigma$-derivations of triangular algebras.

\begin{theorem} \label{descgender}
Let $\T = \Tri(A, M, B)$ be a triangular algebra and $\sigma$ an automorphism of $\T$. Assume that the algebras $A$ and $B$ have only trivial idempotents. Then there exist an $f_\sigma$-derivation $d_A$ of $A$, a $g_\sigma$-derivation $d_B$ of $B$, and a linear map $\xi: M \to M$ such that 
\begin{align*} 
\xi(am) & = d_A(a)m + f_\sigma(a)\xi(m), 
\\
\xi(mb) & = \xi(m)b + \nu_\sigma(m) d_B(b), 
\end{align*} 
for all $a \in A, \, b \in B, \, m \in M$, which satisfy that 
\begin{equation*} 
D 
\left( 
\begin{array}{cc}
a & m  \\
  & b
\end{array}
\right) = 
\left( 
\begin{array}{ccccc}
D_A(a)  &&&&  f_\sigma(a)m_d + m_D b - m_\sigma D_B(b) +
\\
&&&& \xi(m) + D_A(1_A)m
  \\
  &&&&
  \\
  &&&& D_B(b)
\end{array}
\right),
\end{equation*} 
where $m_\sigma$, $m_d$, $m_D$ are elements of $M$, $D_A$ is a generalized $f_\sigma$-derivation of $A$ with associated $f_\sigma$-derivation $d_A$, and $D_B$ is a generalized $g_\sigma$-derivation of $B$ with associated $g_\sigma$ derivation $d_B$. In other words:
\begin{align*}
D_A(aa') & = D_A(a)a' + f_\sigma(a)d_A(a'),
\\
D_B(bb') & = D_B(b)b' + g_\sigma(b)d_B(b'),
\end{align*}
for all $a, a' \in A$ and $b, b' \in B$.
\end{theorem}

\begin{proof}
Let $D$ be a generalized $\sigma$-derivation of $\T$. Then there exists a $\sigma$-derivation $d$ of $\T$ such that
\begin{equation} \label{eqgerder}
D(xy) = D(x)y + \sigma(x)d(y), \quad \forall \, x, y \in \T.
\end{equation}
From Theorem \ref{sigmader} we find an $f_\sigma$-derivation $d_A$ of $A$, a $g_\sigma$-derivation of $B$, and a linear map $\xi: M \to M$ such that 
\begin{align*} 
\xi(am) & = d_A(a)m + f_\sigma(a)\xi(m), 
\\
\xi(mb) & = \xi(m)b + \nu_\sigma(m) d_B(b), 
\end{align*} 
for all $a \in A, b \in B, m \in M$. Moreover, $d$ can be written as:
\begin{equation*} 
d \left(
\begin{array}{cc}
a & m  \\
  & b 
\end{array}
\right) = 
\left(
\begin{array}{ccccc}
d_A(a) &&&& f_\sigma(a)m_d - m_d b - 
\\
&&&& m_\sigma d_B(b) + \xi(m)  
\\
&&&&
\\
  &&&& d_B(b) 
\end{array}
\right),
\end{equation*}
for all $\left( 
\begin{array}{cc}
a & m  \\
  & b
\end{array}
\right) \in \T$; where $m_d$ is an element of $M$. Let us write $D$ as:
\[
D 
\left( 
\begin{array}{cc}
a & m  \\
  & b
\end{array}
\right) = 
\left( 
\begin{array}{cccc}
\delta_1(a) + \delta_2(m) + \delta_3(b) &&& \tau_1(a) + \tau_2(m) + \tau_3(b)
 \\
 &&&
 \\
  &&& \mu_1(a) + \mu_2(m) + \mu_3(b)
\end{array}
\right),
\] 
for all $\left( 
\begin{array}{cc}
a & m  \\
  & b
\end{array}
\right) \in \T$; where $\delta_1: A \to A$, $\delta_2: M \to A$, $\delta_3: B \to A$, $\tau_1: A \to M$, $\tau_2: M \to M$, $\tau_3: B \to M$, $\mu_1: A \to B$, $\mu_2: M \to B$ and $\mu_3: B \to B$ are linear maps. 

Given $m \in M$, let us start by applying \eqref{eqgerder} with the elements 
$x = p$ and $y = \left( 
\begin{array}{cc}
0 & m \\
  & 0
\end{array}
\right)$; it gives that $\delta_2 = \mu_2 = 0$ and that 
\begin{equation} \label{auxiliar0}
\tau_2(m) = \delta_1(1_A)m + \xi(m), \quad \forall \, m \in M.
\end{equation} 
Next, take $a, a' \in A$; apply \eqref{eqgerder} with $x = \left( 
\begin{array}{cc}
a & 0 \\
  & 0
\end{array}
\right)$ and $y = \left( 
\begin{array}{cc}
a' & 0 \\
  & 0
\end{array}
\right)$ to obtain that $\mu_1 = 0$ and that 
\begin{align} 
\delta_1(aa') & = \delta_1(a)a' + f_\sigma(a) d_A(a'), \label{auxiliar1}
\\ \label{auxiliar2}
\tau_1(aa') & = f_\sigma(aa') m_d.
\end{align}
From \eqref{auxiliar1} we get that $\delta_1$ is a generalized $f_\sigma$-derivation of $A$; making $a' = 1_A$ in 
\eqref{auxiliar2} gives that 
\begin{equation} \label{auxiliar3}
\tau_1(a) = f_\sigma(a) m_d, \quad \forall \, a \in A. 
\end{equation}
Given $b, b' \in B$, a use of \eqref{eqgerder} with the elements $x = \left( 
\begin{array}{cc}
0 & 0 \\
  & b
\end{array}
\right)$ and $y = \left( 
\begin{array}{cc}
0 & 0 \\
  & b'
\end{array}
\right)$ shows that $\delta_3 = 0$, and that 
\begin{align} \label{auxiliar4}
\tau_3(bb') & = \tau_3(b)b' - m_\sigma g_\sigma(b)d_B(b'),
\\ \label{auxiliar5}
\mu_3(bb') & = \mu_3(b)b'+ g_\sigma(b) d_B(b').
\end{align}
From \eqref{auxiliar4} we obtain that 
\begin{align} \label{auxiliar6}
\tau_3(b)  = \tau_3(1_B)b - m_\sigma d_B(b), \quad \forall \, b \in B.
\end{align}
On the other hand, note that \eqref{auxiliar5} allows us to conclude that $\mu_3$ is a generalized $g_\sigma$-derivation of $B$. To finish, make $D_A := \delta_1$, $D_B :=  \mu_3$, $m_D :=  \tau_3(1_B)$, and apply \eqref{auxiliar0}, \eqref{auxiliar3} and \eqref{auxiliar6} to get that
\[
\tau_1(a) + \tau_2(m) + \tau_3(b) = f_\sigma(a)m_d + m_D b - m_\sigma d_B(b) + D_A(1_A)m + \xi(m), 
\]
as desired.
\end{proof}

As direct consequences of Theorem \ref{descgender} we obtain the following results:

\begin{corollary} \label{descripgender}
Let $\T = \Tri(A, M, B)$ be a triangular algebra. Then every generalized derivation $D$ of $\T$ is of the following form:
\begin{equation*} 
D 
\left( 
\begin{array}{cc}
a & m  \\
  & b
\end{array}
\right) = 
\left( 
\begin{array}{ccccc}
D_A(a)  &&&&  am_d + m_D b + D_A(1_A)m 
  \\
  &&&&
  \\
  &&&& D_B(b)
\end{array}
\right),
\end{equation*}
where $m_d$, $M_D$ are elements of $M$, $D_A$ is a generalized derivation of $A$, and $D_B$ is a generalized derivation of $B$.
\end{corollary}

\begin{corollary} \label{descrileftmulti}
Let $\T = \Tri(A, M, B)$ be a triangular algebra. Then every left multiplier $F$ of $\T$ is of the following form:
\begin{equation*} 
F 
\left( 
\begin{array}{cc}
a & m  \\
  & b
\end{array}
\right) = 
\left( 
\begin{array}{ccccc}
F_A(a)  &&&&  m_F b + F_A(1_A)m 
  \\
  &&&&
  \\
  &&&& F_B(b)
\end{array}
\right),
\end{equation*}
where $m_F$ is an element of $M$, $F_A$ is a left multiplier of $A$, and $F_B$ is a left multiplier of $B$.
\end{corollary}


\section{Theorems of Posner and Mayne for triangular algebras}Ê\label{thms}

We examine the validity of the classical results due to Posner \cite{Po} and Mayne \cite{M} in the context of $\sigma$-maps of triangular algebras.

\begin{theorem}
Let $\T = \Tri(A, M, B)$ be a triangular algebra and $\sigma$ an automorphism of $\T$. Assume that the algebras $A$ and $B$ have only trivial idempotents. If a $\sigma$-derivation $d$ of $\T$ is $\sigma$-centralizing, then $d = 0$.
\end{theorem}

\begin{proof}
Apply \eqref{sgder} to write $d$ as follows:
\begin{equation*} 
d \left(
\begin{array}{cc}
a & m  \\
  & b 
\end{array}
\right) = 
\left(
\begin{array}{ccccc}
d_A(a) &&&& f_\sigma(a)m_d - m_d b - 
\\
&&&& m_\sigma d_B(b) + \xi(m)  
\\
&&&&
\\
  &&&& d_B(b) 
\end{array}
\right),
\end{equation*}
where $d_A$ is an $f_\sigma$-derivation of $A$, $d_B$ is a $g_\sigma$-derivation of $B$, $m_d$ is an element of $M$ and $\xi: M \to M$ is a linear map which satisfies
\begin{equation} \label{xiproper}
\xi(am) = d_A(a)m + f_\sigma(a)\xi(m), \quad \quad \quad \xi(mb) = \xi(m)b + \nu_\sigma(m) d_B(b),
\end{equation} 
for all $a \in A$, $b \in B$, $m \in M$. On the other hand, since $d$ is $\sigma$-centralizing we have that 
$[x, d(x)]_\sigma \in \Z(\T)$, for all $x \in \T$. Linearizing the expression above yields that 
$[x, d(y)]_\sigma + [y, d(x)]_\sigma \in \Z(\T)$, for all $x, y \in \T$. Taking $x = p$ and $y = \left( 
\begin{array}{cc}
0 & m \\
  & 0
\end{array}
\right)$, we obtain that 
$\left( 
\begin{array}{ccc}
0 && \xi(m)  \\
\\
  && 0
\end{array}
\right) \in \Z(\T)$, which implies that $\xi(m) = 0$, for all $m \in M$. From \eqref{xiproper} we get that $d_A(a)m = 0$ and $\nu_\sigma(m) d_B(b) = 0$, for all $a \in A$, $m \in M$, $b \in B$. Applying that $M$ is a faithful $(A, B)$-module and that $\nu_\sigma$ is bijective we get that $d_A = d_B = 0$. To finish, it remains to show that $m_d = 0$. From $[p, d(p)]_\sigma \in \Z(\T)$ we have that 
$\left( 
\begin{array}{ccc}
0 && m_d  \\
\\
  && 0
\end{array}
\right) \in \Z(\T)$, which yields that $m_d = 0$, as desired.
\end{proof}

\begin{corollary}
If $d$ is a centralizing derivation of a triangular algebra, then $d = 0$.
\end{corollary}

In \cite{RSO} it has been proved that the identity is the only commuting automorphism of a triangular algebra. In what follows, we will generalize this result by considering centralizing maps.

\begin{theorem}
Let $\T = \Tri(A, M, B)$ be a triangular algebra such that the algebras $A$ and $B$ have only trivial idempotents. If $\sigma$ is a centralizing automorphism of $\T$, then $\sigma = \Id_\T$. 
\end{theorem}

\begin{proof}
Let $\sigma$ be a centralizing automorphism of $\T$. By Theorem \ref{aut} we can write $\sigma$ as follows:
\[ 
\sigma \left(
\begin{array}{cc}
a & m  \\
  & b 
\end{array}
\right) = 
\left(
\begin{array}{ccc}
f_\sigma(a) && f_\sigma(a)m_\sigma - m_\sigma g_\sigma(b) + \nu_\sigma(m)  
\\
&&
\\
  && g_\sigma(b) 
\end{array}
\right),
\]
for all $\left(
\begin{array}{cc}
a & m  \\
  & b 
\end{array}
\right) \in \T$; 
where $f_\sigma$, $g_\sigma$ are automorphisms of $A$, $B$, respectively, $m_\sigma$ is an element of $M$ and $\nu_\sigma: M \to M$ is a linear bijective map such that $\nu_\sigma(am) = f_\sigma(a)\nu_\sigma(m)$ and $\nu_\sigma(mb) = \nu_\sigma(m) g_\sigma(b)$, for all $a \in A$, $b \in B$, $m \in M$. 

An application of \eqref{ce2} with 
$x = \left(
\begin{array}{cc}
a & 0  \\
  & 0 
\end{array}
\right)$ and 
$y = \left(
\begin{array}{cc}
0 & m  \\
  & 0 
\end{array}
\right)$ shows that  
\begin{equation} \label{auxrelacion}
a \nu_\sigma(m) - f_\sigma(a) m = 0, \quad \forall \, a \in A, \, m \in M.
\end{equation}
Set $a = 1$ in \eqref{auxrelacion} to find that $\nu_\sigma$ is the identity map on $M$. Then \eqref{auxrelacion} becomes $(a - f_\sigma(a))M = 0$, which implies that $f_\sigma = \Id_A$ since $M$ is a faithful left $A$-module.

Given $m \in M$ and $b \in B$ we have that 
\[
mb = \nu_\sigma(mb) = \nu_\sigma(m)g_\sigma(b) = m g_\sigma(b),
\]
which yields $g_\sigma = \Id_B$, since $M$ is a faithful right $B$-module. To finish, it remains to show that $m_\sigma = 0$.  
To this end, apply \eqref{ce1} with $x = p$. 
\end{proof}

In view of the previous section, a natural question arises: what can we say about centralizing generalized derivations of triangular algebras? We close the paper by proving the following result.

\begin{theorem}
Let $\T = \Tri(A, M, B)$ be a triangular algebra. Then every generalized derivation of $\T$ which is centralizing
degenerates to a left multiplier.
\end{theorem}

\begin{proof}
Let $D$ be a generalized derivation of $\T$ with associated derivation $d$. By Corollary \ref{descrider} 
we have that 
\begin{equation*} 
d \left(
\begin{array}{cc}
a & m  \\
  & b 
\end{array}
\right) = 
\left(
\begin{array}{ccc}
d_A(a) && am_d - m_d b + \xi(m)  
\\
&&
\\
  && d_B(b) 
\end{array}
\right),
\end{equation*}
where $d_A$ is a derivation of $A$, $d_B$ is a derivation of $B$, $m_d$ is an element of $M$ and $\xi: M \to M$ is a linear map, which satisfies 
\begin{equation} \label{properxi}
\xi(am) = d_A(a)m + a\xi(m), \quad \quad \quad \xi(mb) = \xi(m)b + m d_B(b),
\end{equation} 
for all $a \in A$, $b \in B$, $m \in M$. On the other hand, Corollary \ref{descripgender} allows us to write $D$ as follows:
\[ 
D\left( 
\begin{array}{cc}
a & m  \\
  & b
\end{array}
\right) = 
\left( 
\begin{array}{ccccc}
D_A(a)  &&&&  am_d + m_D b + D_A(1_A)m
  \\
  &&&&
  \\
  &&&& D_B(b)
\end{array}
\right),
\] 
$
\left(
\begin{array}{cc}
a & m  \\
  & b
\end{array}
\right) \in \T$, where $D_A$, $D_B$, $m_d$ and $m_D$ are given as in Corollary \ref{descripgender}. Next by Corollary \ref{condicomm} we have that $D$ is commuting. Keeping 
this in mind, take $a \in A$ and apply Corollary \ref{centra} (iii) to see that 
\[
D_A(a)m = aD_A(1_A)m = D_A(1_A)am, \quad \forall \, \, m \in M
\]
which implies that $D_A(a) = D_A(1_A)a$, since $M$ is a faithful left $A$-module. On the other hand, given $a, a' \in A$ we have that 
\[
D_A(aa') = D_A(1_A)aa' = D_A(a) a', \quad \quad D_A(aa') = D_A(a) a' + a d_A(a').
\]
From this, we find that $D_A$ is a left multiplier of $A$ and that $d_A = 0$. Reasoning as above, applying now Theorem \ref{descricent} (iv), we get that $D_B$ is a left multiplier of $B$ and that $d_B = 0$. From \eqref{properxi} we obtain that $\xi = 0$.

On the other hand, since $D$ is commutating we have that 
\begin{equation} \label{com1}
[x, D(x)]Ê= 0, \quad \forall \, x \in \T \tag{com1}.
\end{equation}
Applying \eqref{com1} first with $x = p$ and secondly with $x = q$ we obtain that $m_d = M_D = 0$. Corollary \ref{descrileftmulti} finishes the proof.  
\end{proof}


\section*{Acknowledgements}
The author was supported by the MINECO through the project MTM2010-15223, and by the Junta de Andaluc\'ia and condos FEDER through the project FQM 336.


\end{document}